\documentclass[12pt]{article}
\usepackage{amsmath,amssymb, MnSymbol, amsthm, stmaryrd, MnSymbol}
\usepackage{amsfonts}
\usepackage{graphicx}
\usepackage{color}
\usepackage{enumerate}

\usepackage[textsize=small]{todonotes}

\newtheorem{thm}{Theorem}[section]

\newtheorem{prop}[thm]{Proposition}
\newtheorem{lem}[thm]{Lemma}
\newtheorem{fact}[thm]{Fact}
\newtheorem{rem}[thm]{Remark}

\theoremstyle{definition}
\newtheorem{dfn}[thm]{Definition}
\newtheorem{exm}[thm]{Example}

\begin{document}

\title{ Probability Logic:\\ A Model Theoretic Perspective}
\author{M. Pourmahdian\footnote{School of Mathematics, Institute for Research in Fundamental Sciences (IPM) and Department of Mathematics and Computer Science, Amirkabir University of Technology, Tehran, Iran.  E-mail: pourmahd@ipm.ir.}\\R. Zoghifard\footnote{School of Mathematics, Institute for Research in Fundamental Sciences (IPM), Tehran, Iran.   E-mail: r.zoghi@gmail.com.}}

\date{}
\maketitle

\begin{abstract}
	In this paper (propositional) probability logic ($PL$)  is investigated  from model theoretic point of view. First of all, the ultraproduct construction is adapted for $\sigma$-additive probability models, and subsequently when this class of models is considered  it is shown that the compactness property holds  with respect to a fragment of $PL$ called basic probability logic ($BPL$). On the other hand, when dealing with finitely-additive probability models, one may extend the compactness property for a larger fragment of probability logic, namely positive probability logic ($PPL$).  We finally prove that while the L\"owenheim-Skolem number of the class of  $\sigma$-additive probability models is uncountable, it is $\aleph_0$ for the class of finitely additive  probability models.	
\end{abstract}


\section{Introduction and Preliminaries }
Propositional probability logic ($PL$) is a framework for specifying  and analyzing properties of  structures involving probability, e.g. probability spaces or Markov processes. This logic provides rules of reasoning about these structures. This natural logic is a modal logic in which bounds on probability are treated as modal operators. So, for each $\alpha\in\mathbb{Q}\cap [0,1]$, the language of $PL$ includes a modal operator $L_{\alpha}$ interpreted as (an agent) assigns probability at least $\alpha$. This logic is shown to be useful in many research areas such as theoretical computer science, economics and philosophy. For example this logic might be used to reason about behavior of a program under probabilistic assumptions about inputs.  Also one can highlight how probability systems and particular probability logic play a crucial rule in game theory.
A type space is an example of a probabilistic system, introduced by Harsanyi in \cite{harsan:games68}, provides an implicit description of beliefs in games with incomplete information. So, in particular $PL$ is known to be useful for studying of type spaces.

This logic is studied from different perspectives. There is a rich source of papers involving  full axiomatization of this logic, aiming to show different type of completeness results, \cite{fagin:logic90,heifmon:prob01,zhou:comdeduc09,zhou:finad09}.  A coalgebraic point of view is another source of research in this area, \cite{goldblatt2010deduction,moss2004harsanyi}.  The aim of this paper is to study this logic from model theoretic perspectives.

To be able to state our results in technical terms we review basic concepts of $PL$.

Assume that   $\mathcal{P}$ is a countable set of propositional variables.
The syntax of probability logic is obtained by adding countable probability modal operators $L_r,M_r$ for each $r\in\mathbb{Q}\cap[0,1]$ to propositional logic. When applying the operator $L_r$ to a formula $\varphi$, then $L_r \varphi$ is interpreted as ``the formula $\varphi$ has probability at least $r$".
In the same way, the intended meaning of formula $M_r \varphi$ is ``the formula $\varphi$ has probability at most $r$".

\begin{dfn}
Formulas of probability logic ($PL$)  is defined by the following grammar:
\begin{eqnarray*}
PL &:=& p\ |\ \neg \varphi\ |\ \varphi\wedge\varphi\ |\ \varphi\vee\varphi\ |\ L_r\varphi\ |\ M_r\varphi,
\end{eqnarray*}
where $p\in \mathcal{P}$ and $r\in\mathbb{Q}\cap[0,1]$.
\end{dfn}
In the following we also consider two fragments of probability logic, namely, \textit{basic probability logic} and \textit{positive probability logic}.

\begin{dfn} 	
	The basic probability logic ($BPL$) and respectively  positive  probability logic ($PPL$) are defined by the following grammars: 
  \begin{eqnarray*}
	 BPL &:=& p\ |\ \neg p\ |\ \varphi\wedge\varphi\ |\ \varphi\vee\varphi\ |\ L_r\varphi.\\
	PPL &:=& BPL \ |\ M_r\varphi.
\end{eqnarray*}
\end{dfn}

Note that in both $BPL$ and $PPL$ the negation only applies to elements of the set $\mathcal{P}$. Furthermore, $BPL$ is a proper fragment of $PPL$ in which applying the modal operators $L_r$ is only allowed.  Note that  $BPL\subsetneq PPL\subsetneq PL$.  
We also name these fragments as probability logics.

The other logical connectives $(\vee,\ \rightarrow,\ \leftrightarrow)$ have their standard definitions.

To interpret the formulas in $PL$ we have to define the notion of probability models.

For any measurable space $(\Omega,\mathcal{A})$ let $\Delta(\Omega, \mathcal{A})$ be the  measurable space of all $\sigma$-additive probability measures on $\Omega$ whose $\sigma$-algebra generated by the sets
\begin{center}
	$\{\mu \in \Delta(\Omega,\mathcal{A}) \ |\ \mu(E)\geq r \}$ for all $E\in \mathcal{A}$ and $r\in \mathbb{Q} \cap[0,1]$.
\end{center}

\begin{dfn}
	A  type space over  a   measurable space $(\Omega,\mathcal{A})$,  is a triple  $\mathfrak{M}=(\Omega, \mathcal{A}, T)$ where
	$T$ is a measurable function from $\Omega$ to $\Delta(\Omega, \mathcal{A})$.
\end{dfn}

$\Omega$ and $T$ are respectively called a \textit{set of states (or possible worlds)} and a \textit{type function}. It follows from the above definition, for $w\in\Omega$, $T(w)$ defines a probability measure on the $\sigma$-algebra $\mathcal{A}$. Furthermore, its measurability indicates that for each $E\in\mathcal{A}$ and  $r\in \mathbb{Q} \cap[0,1]$,  $$\{\omega\in \Omega\ |\  T(\omega,E)\geq r\}\in \mathcal{A}.$$ 

Type spaces are regarded as semantical devices for probability logics. 
We show that  probability logic and its fragments introduced above  have  different model theoretic features.

\begin{dfn}
A probability model is a tuple $\mathfrak{M}=(\Omega, \mathcal{A},T,v)$ where the triple $(\Omega, \mathcal{A},T)$ is a type space and $v:\mathcal{P}\rightarrow \mathcal{A}$ is a valuation function which assigns to each proposition $p\in\mathcal{P}$ the measurable set $v(p)\in\mathcal{A}$. 
\end{dfn}

\begin{dfn}
	A finitely additive probability model is a tuple $\mathfrak{M}=(\Omega, \mathcal{A}, T:\Omega\times \mathcal{A}\rightarrow [0,1], v:\mathcal{P}\rightarrow \mathcal{A})$ where
	\begin{itemize}
		\item $\mathcal{A}$ is an algebra over $\Omega\neq \emptyset$.
		
		\item For each $w\in \Omega$, $T(w)$ defines a finitely additive measure on $\mathcal{A}$. 	
		\item   $T$ is a measurable function, i.e. for each $E\in \mathcal{A}$ and 
		$r\in \mathbb{Q}\cap [0,1]$, $\{w\in\Omega\ |\  T(w)(E)\geq r\}\in\mathcal{A}$.
	\end{itemize}
\end{dfn}

	A model $\mathfrak{M}$ with a distinguished point $w\in\Omega$ is called a pointed probability model and denoted by $(\mathfrak{M},w)$.

Denote the class of pointed probability models by $\mathcal{PM}$ and respectively the class of finitely additive probability pointed models by $\mathcal{FPM}$. Notice that $\mathcal{PM}\subsetneq \mathcal{FPM} $.

\begin{dfn}
For a pointed (finitely additive) probability model $(\mathfrak{M},w)$ and a formula $\varphi\in PL$ the satisfaction relation $\mathfrak{M},w\models \varphi$ is defined inductively in the usual way for propositional variables and boolean connectives. For $L_r, M_r$ operators, if we assume $\llbracket\varphi\rrbracket_{\mathfrak{M}}= \{w\in \Omega \ |\ \mathfrak{M},w\models \varphi\}$, then 
$$\mathfrak{M},w\models L_r\varphi \ \ \text{ if and only if } \ \ T(w)(\llbracket\varphi\rrbracket_{\mathfrak{M}})\geq r$$
and 
$$\mathfrak{M},w\models M_r\varphi \ \ \text{ if and only if } \ \ T(w)(\llbracket\varphi\rrbracket_{\mathfrak{M}})\leq r.$$
\end{dfn}

We often omit the subscript $\mathfrak{M}$ from  $\llbracket\varphi\rrbracket_{\mathfrak{M}}$ and write $\llbracket\varphi\rrbracket$  when no confusion can arise.
Note that, by definition of $v$ and measurability of $T$, it is easy to see that $\llbracket\varphi\rrbracket\in\mathcal{A}$ for any formula $\varphi$. Notice  that for any formula $\varphi$ and $r\in\mathbb{Q}\cap [0,1]$, $\mathfrak{M},w\nmodels  L_r\varphi$ if and only if $\mathfrak{M},w\models  M_s\varphi$, for some $s<r$. A similar statement holds for $M_r \varphi$. This means that the negation of positive formulas can be defined by an infinite disjunction of positive formulas.

The other syntactical and semantical components of probability logics can be defined in the usual way. In particular, if $\mathcal{L}\in \{PL, BPL, PPL\}$ then any set of $\mathcal{L}$-formulas is called an $\mathcal{L}$-theory. Let $\mathcal{K}$ be a subclass of $\mathcal{FPM}$. An $\mathcal{L}$-theory $T$ is satisfiable in $\mathcal{K}$ if there exists a pointed model $(\mathfrak{M},w)\in \mathcal{K}$ such that $ \mathfrak{M},w\models \varphi$, for each $\varphi\in T$. Likewise, $T$ is finitely satisfiable in $\mathcal{K}$ if every finite subset of $T$ is satisfiable in $\mathcal{K}$.  An $\mathcal{L}$-theory $T$ is maximally satisfiable (respectively maximally finitely satisfiable) if it is maximal in the poset of satisfiable $\mathcal{L}$-theories (respectively finitely satisfiable $\mathcal{L}$-theories) ordered by inclusion relation. We say that the logic $\mathcal{L}$ has the compactness property with respect to class $\mathcal{K}$ if an $\mathcal{L}$-theory $T$ is satisfiable in $\mathcal{K}$ if and only if $T$ is finitely satisfiable in $\mathcal{K}$. 
In this paper the following results are established.

	\begin{enumerate}	
	\item  (Theorem \ref{bpl compact})	 $BPL$ has the compactness property with respect to $\mathcal{PM}$. 
	
	\item (Theorem \ref{ppl compact})	 $PPL$ has the compactness property with respect to $\mathcal{FPM}$. 
\end{enumerate}

We also study the  L\"{o}wenheim-Skolem number of the class of probability and finitely additive probability models. Recall that the L\"{o}wenheim-Skolem number of a  class $\mathcal{K}$ of logic $\mathcal{L}$ is the least infinite cardinal $\kappa$ such that every satisfiable $\mathcal{L}$-theory  has a model of size at most $\kappa$.

	\begin{enumerate}
		\item  (Theorem \ref{LS PL})
		Let $\lambda$ be the L\"{o}wenheim-Skolem number of probability models with respect to probability logic. Then $\aleph_0 <\lambda \leq 2^{\aleph_0}$. 
		\item (Theorem \ref{LS fad})
		The L\"{o}wenheim-Skolem number of the class of finitely additive probability models with respect to probability logic is $\aleph_0$. 
	\end{enumerate}

\section{Compactness Property}

In this section we study the compactness property for $BPL$ and $PPL$.  

\subsection{Compactness for $BPL$ }\label{comp sec}
In this subsection we show that the basic probability logic, $BPL$, satisfies the compactness property with respect to $\mathcal{PM}$.  
It is known that  the compactness property does not hold for (full) probability logic. To see this consider the $PL$-theory
\begin{equation}
T=\{ L_{\frac{1}{2}-\frac{1}{4^{n+1}}}p\  |\ n\in \mathbb{N}\} \cup \{\neg L_{\frac{1}{2}}p\}.
\end{equation}
It is easy to show that  $T$ is finitely satisfiable but not satisfiable.

The failure of compactness in $PL$ is partly due to this fact that by using the negation one can express the strict inequality. However, as $BPL$ only applies the negation on propositions the above example is not a $BPL$-theory. 
In fact then restrict ourselves to $BPL$ we will see that the compactness holds in this logic.
To achieve this, we adapt the ultraproduct construction  for probability models and show that the \L o\'s theorem  holds for basic formulas, which means that $BPL$ enjoys the compactness property. We recall first some primary notions related to the ultraproduct construction.

Let $U$ be an ultrafilter over a non-empty set $I$ and $(a_i)_{i\in I}$ be a sequence of elements from $\mathbb{R}$, the set of real numbers. The $U$-limit of this sequence, denoted by $\lim_U a_i$,  is an element $r\in \mathbb{R}$ such that for every $\epsilon>0$ we have $\{ i\in I\ |\ |a_i-r|<\epsilon\}\in U$. It is known that each bounded sequence of elements of $\mathbb{R}$ has a unique limit over each ultrafilter. Furthermore, if $(a_i)_{i\in I}$ is  a bounded sequence of elements of $\mathbb{R}$ and $U$ is an ultrafilter over $I$ then
\begin{itemize}
\item  If $\{i\in I\ | \ a_i \geq r\} \in U$, then $\lim_U a_i\geq r$.
\item  If $\lim_U a_i> r$, then $\{ i\in I\ |\ a_i> r\}\in U$.
\item  $\lim_U a_i \geq r$ if and only if for every $r'<r$ we have $\{i\in I\ |\ a_i\geq r'\}\in U$.
\end{itemize}

For a family   $\langle \Omega_i : i\in I\rangle$ of sets indexed by $I$, let $\prod_{i\in I} \Omega_i$  
be  the Cartesian product of this family defined as the set $\prod_{i\in I} \Omega_i = \{(w_i)_{i\in I}\; |\; w_i\in \Omega_i\}$. Two elements $(w_i)_{i\in I}$ and $(v_i)_{i\in I}$ in $\prod_{i\in I} \Omega_i$ are $U$-equivalent, denoted by $(w_i)_{i\in I} \sim_U (v_i)_{i\in I}$, if $\{i\in I\; |\; w_i=v_i\}\in U$. Clearly $\sim_U$ defines an equivalence relation on $\prod_{i\in I}\Omega_i$. Let $(w_i)_U$ be the equivalence class of $(w_i)_{i\in I}$,  and the resulting $\prod_U \Omega_i$ be the set of all equivalence classes. 

Now suppose $(\Omega_i,\mathcal{A}_i,T_i)_{i\in I}$ is a family of type spaces.
Then for each sequence $(A_i)_{i\in I}$ with $A_i\in \mathcal{A}_i$ set  
$$(A_i)_U=\{(w_i)_U\in \prod_U\Omega_i\; |\; \{i\in I\; |\; w_i\in A_i\}\in U\}$$
 and let $\mathcal{A}=\{ \ (A_i)_U \ |\ A_i\in \mathcal{A}_i\}.$ It is easy to see that $\mathcal{A}$ forms a boolean algebra over $\prod_U \Omega_i$.
Now, define the function $T':\prod_U \Omega_i\times \mathcal{A}\rightarrow [0,1]$ as  follows:
$$T'((w_i)_U)((A_i)_U) = \lim_U T_i(w_i)(A_i).\;\;\;\;\;\;\;\;\;\;\  (*)$$

Note that since for each $(w_i)_{i\in I}$ the sequence $(T_i(w_i))_{i\in I}$ is  a bounded sequence of real numbers, the $U$-limit $\lim_U T_i(w_i)(A_i)$ exists. Moreover,  if $(w_i)_{i\in I} \sim_U (v_i)_{i\in I}$ then for each $(A_i)_U\in \mathcal{A}$ we have $\lim_U T_i(w_i)(A_i)=\lim_U T_i(v_i)(A_i)$.
Also 
$(A_i)_U=(B_i)_U$ implies that 
$\{i\in I\ |\ A_i=B_i\}\in U$.  
 Hence $(*)$ is well-defined.

\begin{lem}
	For each $(w_i)_U\in \prod_U \Omega_i$, the function $T'((w_i)_U)(.)$ is  a premeasure on the boolean algebra $\mathcal{A}$.
\end{lem}
\begin{proof}
	Fix $(w_i)_U\in \prod_U\Omega_i$.
	First note that
	$$T'((w_i)_U)(\emptyset)=\lim_U T_i(w_i)(\emptyset)=0.$$
	Now we have to show that whenever $\{(A^j_i)_U\ |\ j\in \mathbb{N}\}$ is a countable family of disjoint members of $\mathcal{A}$ if
	$\bigcup_{j\in \mathbb{N}} (A^j_i)_U\in\mathcal{A}$,
	then
	$$T'((w_i)_U)(\bigcup_{j\in \mathbb{N}} (A^j_i)_U)=\sum_{j\in \mathbb{N}} T'((w_i)_U)((A^j_i)_U).$$

		By Fact \ref{premeas equiv},  we prove that for each decreasing sequence
		$(A^0_i)_U\supseteq (A^1_i)_U\supseteq \dots$
		of elements of $\mathcal{A}$, if
		$\bigcap_{j} (A^j_i)_U=\emptyset$
		then 
		$\lim_{j\rightarrow\infty} T'((w_i)_U)(A^j_i)=0$.

		Now suppose on the contrary that 
		$\lim_{j\rightarrow\infty} T'((w_i)_U)(A^j_i)> 0$.
		So, there exists $\epsilon >0$ such that  for all $j\in\mathbb{N}$ we have
		$\lim_U T_i(w_i)(A^j_i)> \epsilon$.
		So 
		$I_j=\{i\in I\ |\ T_i(w_i)(A_i^j)\geq\epsilon\}\in U$,
		 for all $j\in\mathbb{N}$.
		
		Now we show that there is  a decreasing sequence $((B^j_i)_{j\in\mathbb{N}})_U$ such that
		$(A^j_i)_U=(B^j_i)_U$
		and for some $i\in I$,
		$\bigcap_j B_i^j=\emptyset$
		, while
		$T_i(w_i)(B_i^j)\geq \epsilon$, for each $j\in\mathbb{N}$.
		
		Since 
		$(A_i)_U^j\supseteq (A_i)_U^{j+1}$,
		it follows that 
		$\{i\in I\ |\ A_i^j\supseteq A_i^{j+1}\}\in U$, for each $j\in\mathbb{N}$. 
		
		Thus for each
		$j\in\mathbb{N}$
		we have
		$$S_j = \{i\in I\ |\ A_i^0\supseteq A_i^j\} \cap \{i\in I\ |\ T_i(w_i)(A_i^j)\geq \epsilon\}\in U.$$
		Now define the sets $B_i^j$ as follows:
		$$B_i^j=\left\{
		\begin{array}{ll}
		A_i^0 & \text{ if } j=0 \text{ or }  i\not\in S_j \text{ and } k=min_n(i\in S_n),\\
		A_i^j & \text{ if } j>0 \text{ and } i\in S_j.
		\end{array}
		\right.$$
		Hence, $\{i\in I\ |\ B_i^j\supseteq B_i^{j+1}\}\in U$ for each  $j\in\mathbb{N}$. Also,
		$T_i(w_i)(B^j_i)\geq\epsilon$
		for each  $j\in\mathbb{N}$ and each  $i\in I_0$.
		Furthermore, by definition of $B_i^j$s we have
		$(B^j_i)_U=(A^j_i)_U$.
		So $\bigcap_j (B^j_i)_U=\emptyset$.
		Since  $I_0\in U$, there  exists  $i\in I_0$ such that
		$\bigcap_j B_i^j=\emptyset$.
		But this contradicts  $T_i(w_i)(B_i^j)\geq\epsilon$. 
\end{proof}

Therefore, for each $(w_i)_U\in \prod_U \Omega_i$, the function $T'((w_i)_U)(.)$ is  a premeasure on $\mathcal{A}$.  So by Fact \ref{caratheo}, it could be  extended to the measure $T((w_i)_U)(.)$ on $\mathcal{A}_U=\sigma(\mathcal{A})$.
Now to prove that $T$ is a type function on $(\prod_U \Omega_i, \mathcal{A}_U)$ we have to show that it is measurable, i.e.
$$\{(w_i)_U\ |\ T((w_i)_U)(E)\geq\alpha\}\in\mathcal{A}_U,$$
for each $E\in\mathcal{A}_U$  and $\alpha\in\mathbb{Q}\cap [0,1]$.

\begin{lem}\label{meas ultrapro}
	$T$ is a measurable function.
\end{lem}
\begin{proof}

	Let $$\mathcal{B}=\{E\in\sigma(\mathcal{A}) \ |\ \text{for each}\  r\in\mathbb{Q}\cap [0,1],\ \{ w \ |\  T(w)(E)\geq r\}\in\sigma(\mathcal{A})\}.$$ First we show that if $E\in\mathcal{A}$, then $E\in \mathcal{B}$. In this situation, there are
	$E_i\in\mathcal{A}_i$ such that $E=(E_i)_U$.
	So we have
	\begin{align*}
	\{(w_i)_U\ |\ T((w_i)_U)(E)\geq\alpha\} &= \{(w_i)_U\ |\ \lim_U T_i(w_i)(E_i)\geq\alpha\}\\
	&= \{(w_i)_U\ |\ \forall \alpha'<\alpha , \alpha'\in\mathbb{Q}\cap[0,1], \{i\in I\ |\  T_i(w_i)(E_i)\geq\alpha'\}\in U\}\\
	&= \bigcap_{\substack{{\alpha'<\alpha}\\ {\alpha'\in\mathbb{Q}\cap[0,1]}}} \{(w_i)_U\ |\ \{i\in I\ |\ T_i(w_i)(E_i)\geq\alpha'\}\in U\}.
	\end{align*}
	On the other hand for each
	$\alpha'\in\mathbb{Q}\cap[0,1]$ with $\alpha'<\alpha$,
	$$\{(w_i)_U\ |\ \{i\in I\ |\ T_i(w_i)(E_i)\geq \alpha'\}\in U\}=(A_i)_U$$
	where $A_i=\{w_i\ |\ T_i(w_i)(E_i)\geq \alpha'\}$, for each $i\in I$.
	Hence
	$(A_i)_U\in\mathcal{A}$
	and therefore,
	$$\bigcap_{\substack{{\alpha'<\alpha}\\ {\alpha'\in\mathbb{Q}\cap[0,1]}}} \{(w_i)_U\ |\ \{i\in I\ |\ T_i(w_i)(E_i)\geq\alpha'\}\in U\} \in \mathcal{A}_U.$$
    Next we show that if  $E$ is a union of an increasing sequence $E_1\subseteq E_2\subseteq \dots$ of elements of $\mathcal{B}$ then $E\in\mathcal{B}$. 

	Assume that $E_1\subseteq E_2\subseteq \dots$ such that
	$E_j\in \mathcal{A}_U$ and  $E=\bigcup_j E_j$
	and the claim is true for each $E_j$.
	\begin{align*}
	H= \{(w_i)_U\ |\ T((w_i)_U)(E)\geq\alpha\} &= \{(w_i)_U\ |\ T((w_i)_U)(\bigcup_j E_j)\geq\alpha\}\\
	&= \{(w_i)_U\ |\ \lim_{j\rightarrow\infty} T((w_i)_U)(E_j)\geq\alpha\}\\
	&= \{(w_i)_U\ |\ \forall\alpha'<\alpha\ \exists j\ T((w_i)_U)(E_j)\geq \alpha'\}\\
	&= \bigcap_{\substack{{\alpha'<\alpha}\\ {\alpha'\in\mathbb{Q}\cap[0,1]}}} \bigcup_{j=1}^\infty \{(w_i)_U\ |\ T((w_i)_U)(E_j)\geq\alpha'\}.
	\end{align*}
	By induction hypothesis for each $j$ and each  $\alpha'<\alpha$  we have
	$\{(w_i)_U\ |\ T((w_i)_U)(E_j)>\alpha'\}\in \mathcal{A}_U$. Since $\mathcal{A}_U$ is a $\sigma$-algebra, it follows that $H\in \mathcal{A}_U$.
	
	Therefore, $\mathcal{B}$ is a monotone class which includes  the algebra $\mathcal{A}$.
 	So by Fact \ref{monotone class}, $\mathcal{B}=\sigma(\mathcal{A})$ and the proof is complete.
\end{proof}

 Based on the above lemmas, we define ultraproduct  of probability models. 
\begin{dfn}
	Let $\langle\mathfrak{M}_i=(\Omega_i,\mathcal{A}_i,T_i,v_i)\, :\,  i\in I\rangle$ be a family of probability models and $U$ be a non-principal ultrafilter over $I$. The ultraproduct of the family of probability models $\langle\mathfrak{M}_i:\ i\in I\rangle$ over $U$ is a model
	$\mathfrak{M}= \prod_U \mathfrak{M}_i=(\Omega_U,\mathcal{A}_U,T_U,v_U)$ where
	\begin{itemize}
		\item $\Omega_U,\; \mathcal{A}_U$ and $T_U$ are defined as above.
	
		\item $(w_i)_U\in v_U(p)$ if and only if $\{ i\in I\ |\ w_i\in v_i(p)\}\in U$.
	\end{itemize}
\end{dfn}

To ease the notation for each formula $\varphi$ we use $\llbracket\varphi\rrbracket_U$ instead of $\llbracket\varphi\rrbracket_{ \prod_U \mathfrak{M}_i}$.

The following theorem gives a weak version of  the \L o\'s theorem  for basic probability logic. 

\begin{thm}
	Let  $\langle\mathfrak{M}_i\, :\, i\in I\rangle$ be a family of probability models and $U$ be a non-principal ultrafilter over $I$. Suppose $\varphi$ is a basic formula. Then
		$\{i\in I\; |\; \mathfrak{M}_i,w_i\models\varphi\}\in U$ implies $\prod_U\mathfrak{M}_i,(w_i)_U\models\varphi$.
\end{thm}
\begin{proof}
	By induction on the complexity of basic formulas one can show  that for each basic  formula $\varphi$ we have
	$$(\llbracket\varphi\rrbracket_{\mathfrak{M}_i})_U\subseteq \llbracket\varphi\rrbracket_U.$$
	In fact, by definition, for the  atomic formulas and their negations  we have
	$(\llbracket\varphi\rrbracket_{\mathfrak{M}_i})_U= \llbracket\varphi\rrbracket_U$.
	
	It  is also easy to prove the induction step for the  boolean connectives $\wedge$ and $\vee$. Now knowing the induction hypothesis for basic formula $\varphi$ we have the  followings: 
	\begin{align*}
	(w_i)_U\in (\llbracket L_r\varphi\rrbracket_{\mathfrak{M}_i} )_U &\Rightarrow \{ i\in I\ |\ w_i\in \llbracket L_r\varphi\rrbracket_{\mathfrak{M}_i}  \}\in U\\
	&\Rightarrow \{i\in I\ |\ T_i(w_i)(\llbracket\varphi\rrbracket_{\mathfrak{M}_i}) \geq r \} \in U\\
	&\Rightarrow \lim_U T_i(w_i)(\llbracket \varphi\rrbracket_{\mathfrak{M}_i}) \geq r\\
	&\Rightarrow T((w_i)_U)((\llbracket \varphi\rrbracket_{\mathfrak{M}_i})_U) \geq r\\
	&\Rightarrow T((w_i)_U)(\llbracket \varphi\rrbracket_U) \geq r\\
	&\Rightarrow (w_i)_U\in \llbracket L_r\varphi\rrbracket_U.
	\end{align*}
Note that  the fifth line is obtained  from the fourth line by the induction hypothesis.
\end{proof}

The above one directional statement is mainly due to the fundamental fact that $\lim_U(a_i)\geq r$ does not imply that $\{i\in I:\  a_i\geq r\}\in U$ . However the above theorem still enables us to prove the compactness theorem for basic probability logic.  

\begin{thm}[$BPL$-Compactness]\label{bpl compact}
	Suppose that $\Gamma$ is a $BPL$-theory. Then $\Gamma$ is satisfiable in $\mathcal{PM}$ if and only if it is finitely satisfiable.
\end{thm}

We conclude this subsection by giving an example which shows that the compactness fails even for positive probability logic. So, even by avoiding negation and mixing $L_r$ and $M_s$ operators we could find a theory $\Gamma$ which is finitely satisfiable but not satisfiable.

\begin{exm}\label{exm pr}
	Let
	$$\Sigma=\{\ M_0(M_0p\vee L_1p)\ \}\cup \{\ M_{\frac{1}{2}}(L_{\frac{1}{2^i}}p \wedge M_{1-\frac{1}{2^i}}p) \ |\ i\in \mathbb{N}\}.$$
	We show that $\Sigma$  is finitely satisfiable but it is not satisfiable in any probability model.
	
	For each probability model $\mathfrak{M}$, Put
	$$A_0=\{w\in \Omega\ |\ T(w)(\llbracket p\rrbracket)=0\} \ \ \ \text{and} \ \ \ A_0'=\{w\in \Omega\ |\ T(w)(\llbracket p\rrbracket)=1\}.$$
	Also for each $i\in \mathbb{N}$, let
	$$A_i=\{w\in \Omega\ |\ \frac{1}{2^i}\leq T(w)(\llbracket p\rrbracket)\leq 1-\frac{1}{2^i}\}.$$
	For each finite subset $\Sigma'\subseteq \Sigma$, suppose $k$ is the greatest index $i$ such that  $M_{\frac{1}{2}}(L_{\frac{1}{2^i}}p \wedge M_{1-\frac{1}{2^i}}p)\in \Sigma'$.
	Let $\mathfrak{M}$ be a probability model
	$(\Omega=\{w_1,w_2\},\mathcal{P}(\Omega),T,v)$
	such that $v(p)=\{w_1\}$ and
	
\begin{itemize}	
\item 	$T(w_1)(\{w_1\})=T(w_1)(\{w_2\})=\frac{1}{2},$
\item  $0< T(w_2)(\{w_1\}) < \frac{1}{2^k},$
\item  $T(w_2)(\{w_2\})=1-T(w_2)(\{w_1\}).$
	
\end{itemize}	
Therefore, $\mathfrak{M},w_1\models\Sigma'$.

	However, $\Sigma$ is not satisfiable, since otherwise, if $\mathfrak{M},w\models\Sigma$ we have $T(w)(A_0)=T(w)(A_0')=0$ and $T(w)(A_i)\leq \frac{1}{2}$, for all $i\in\mathbb{N}$. As $T(w)$ is a $\sigma$-additive probability measure and $A_1\subseteq A_2\subseteq \dots$, we have
	$T(w)(\bigcup_i A_i)=\lim_{i\rightarrow \infty} T(w)(A_i)\leq \frac{1}{2}$. On the other hand,
	$\Omega=A_0 \cup A_0' \cup \bigcup_{i=1}^\infty A_i = A_0 \cup A_0' \cup \{w'\in \Omega\ |\ 0<T(w')(\llbracket p\rrbracket)<1\}$.
	So, 	$T(w)(\Omega)=T(w)(A_0)+T(w)(A_0')+T(w)(\bigcup_{i=1}^{\infty}A_i)<1$, a contradiction.
\end{exm}

In subsection \ref{compactness sec}, we show that in the above example $\Sigma$ has a finitely additive probability model. In fact when working with the finitely additive probability models the compactness property holds for positive probability logic $PPL$.

\subsection{The Compactness for $PPL$ }\label{compactness sec}

In this subsection we prove that $PPL$ has the compactness property with respect to the class of finitely additive probability models.

As we noted, Example \ref{exm pr} shows that when working with probability models the compactness property fails for $PPL$. So, it is not possible to adapt the ultraproduct construction for finitely additive probability models to prove the $PPL$-compactness. However, we will see that the Henkin method can be implemented to prove this property for finitely additive probability models.
  
In the following we show that the $PPL$-theory $\Sigma$ given in the Example \ref{exm pr} is satisfiable in $\mathcal{FPM}$, i.e. there exists a finitely additive pointed model $\mathfrak{M},w\models \Sigma$.

\begin{exm}
Let $\Sigma$ be a $PPL$-theory as given in Example \ref{exm pr}. 
Define the model  $\mathfrak{M}=(\mathbb{N}, \mathcal{P}(\mathbb{N}), T, v)$ which satisfies the following conditions: 
\begin{itemize}
\item $v(p)=\{0\}$. 
\item  For each $n\neq 0$, $T(n)$ is a $\sigma$-additive measure on $\mathcal{P}(\mathfrak{N})$ with the condition
	\[  T(n)(\{x\})= \left\{
	\begin{array}{ll}
	\frac{1}{2^n} & \text{ if } 0\leq x\leq 2^n-1, \\
	0 & \text{ if } x> 2^n-1. 
	\end{array}
	\right. \]
	
\item For a non-principal ultrafilter $U$ over $\mathbb{N}$, we define $T(0)$ as:
	\[  T(0)(X)= \left\{
	\begin{array}{ll}
	1 & \text{ if } X\in U, \\
	0 & \text{ if }  X\not\in U.
	  \end{array}
	\right. \]
	It is easy to see that $T(0)$ is a finitely additive probability measure.  Note that for every $i\neq0$,
	$\{w\in \mathbb{N}\ |\ \frac{1}{2^i}\leq T(w)(\llbracket p\rrbracket)\leq 1-\frac{1}{2^i}\}=\{1,\dots,i\}\nin U$. So for the point $0$ we have

\end{itemize}
	\begin{align*}
	T(0)(\llbracket p\rrbracket) & = 0,\\
	T(0)(\llbracket M_0 p\rrbracket) & = T(0)(\{0\}) = 0,\\
	T(0)(\llbracket L_1 p\rrbracket) &= T(0)(\emptyset) =0,\\
	T(0)(\llbracket L_{\frac{1}{2^i}}p \wedge M_{1-\frac{1}{2^i}}p\rrbracket) &= 
	T(0)(\{1,\dots,i\})= 0,\\
	T(0)(\{n\ |\ 0<T(n)(\llbracket p\rrbracket)<1 \}) &= T(0)(\mathbb{N}\setminus\{0\}) = 1.
	\end{align*}
Therefore, $\mathfrak{M},0\models\Sigma$.
\end{exm}

To prove  the compactness property for $PPL$-theories, 
 we construct a canonical finitely additive probability model  $\mathfrak{M}_{C}=(\Omega_C,\mathcal{P}(\Omega_C),T_C,v_C)$ whose set of states $\Omega_C$ consists of all maximally finitely satisfiable positive theories such that its satisfaction relation is given by 
$\mathfrak{M}_C, \Gamma\models \varphi$ if and only if $\varphi\in\Gamma$, for each $\Gamma\in\Omega_C$ and positive formula $\varphi$. Lemma \ref{lemwp} is the key  ingredient in proving the truth lemma and is based on the following known fact which relates the satisfiability of a formula $\varphi$ with solvability of certain  finite system of  linear inequalities $S_{\varphi}$, \cite{fagin:logic90,zhou:intui11}.  

Before defining the canonical model we need to recall some basic notions. The (modal or) probability depth of a formula $\varphi$ is the maximum number of nesting probability operators used in $\varphi$. More formally,  
\begin{dfn}
	\begin{itemize}
		\item $\delta(p)=0$, for each atomic formula $p$.
		\item $\delta(\neg\varphi)=\delta(\varphi)$.
		\item $\delta(\varphi\wedge\psi) = \delta(\varphi\vee \psi)=\max(\delta(\varphi),\delta(\psi))$.
		\item $\delta(L_r\varphi)=\delta(M_s \varphi)=\delta(\varphi)+1$.
	\end{itemize}
\end{dfn}

For a formula $\varphi$  , let local language $\mathcal{L}_{\varphi}$ be the largest set of formulas satisfying the following conditions:
\begin{itemize}
	\item The propositional variables are those occur in $\varphi$.
	\item  Each $r\in\mathbb{Q}\cap [0,1]$ appears in a probability operators of a formula in $\mathcal{L}_{\varphi}$  is a multiple of $\frac{1}{q_{\varphi}}$ where $q_{\varphi}$ is the least common multiple of all denominators of the rational numbers appearing in probability operators in $\varphi$.
	\item The depth of formulas in $\mathcal{L}_{\varphi}$ is at most $\delta(\varphi)$.
	
\end{itemize}
It is clear that up to logical equivalence the local language $\mathcal{L}_{\varphi}$ consists of only finitely many formulas.

Now we state the following fact. We refer the reader to \cite{zhou:intui11}, Theorem 3, for the construction of $S_{\varphi}$ as well as its complete proof. We should only point out whenever $\varphi$ is a positive probability  formula, the set $S_{\varphi}$ consists of only closed inequalities, i.e. linear inequalities of form $x_1+\dots+ x_n\geq r$ or $x_1+\dots+ x_n\leq s$. 

\begin{fact}\label{system}
	For any probability formula $\varphi$ there is a system of linear inequalities $S_\varphi$ such that  $\varphi$ is satisfiable if and only if $S_\varphi$ is solvable.
\end{fact}

The following crucial lemma follows from Fact \ref{system} and is needed for characterizing the satisfaction relation of the canonical model.

\begin{lem}\label{lemwp}
	Let $\Gamma$ be a  finite set of positive formulas. If $\Gamma\cup\{L_r\varphi\}$ is not satisfiable, then there is a rational number $r'$ with $0<r'<r$  such that $\Gamma\cup\{L_{r'}\varphi\}$ is not satisfiable.
	Similarly,  if $\Gamma\cup \{M_s\varphi\}$ is not satisfiable, then there is a rational  number $s'$ with $s<s'<1$ such that $\Gamma \cup \{M_{s'}\varphi\}$ is not satisfiable.
	
\end{lem}
\begin{proof}
First of all, without loss of generality we may assume that $\Gamma$ is satisfiable. 	Put $\psi=\bigwedge\Gamma \wedge \varphi$.  Let $\{H_1,\dots, H_n
\}$ be the set of all  maximally satisfiable sets of formulas over $\mathcal{L}_\psi$. Associate to each  $H_i$ a variable $x_i$. Note that each formula in $\mathcal{L}_\psi$ is logically equivalent to a disjunction of (conjunction of) $H_i$'s. So, for every formula $\theta$ in this fragment let $I_{\theta}\subseteq \{1,\dots,n\}$, such that $\theta\equiv \bigvee_{i\in I_{\theta}} \bigwedge H_i$.

	Since $\Gamma$ is a finite set of positive formulas, $\gamma=\bigwedge\Gamma$ is also a positive formula. On the other hand,  the positive formula $\gamma$ is equivalent to a disjunction of (satisfiable) formulas of the form
	$$\gamma_i = \bigwedge_j p_{ij} \wedge \bigwedge_{j'} \neg p_{ij'} \wedge \bigwedge_l L_{r_{il}}\varphi_{il} \wedge \bigwedge_{l'} M_{s_{il'}}\varphi_{il'}.$$
	 $\Gamma\cup\{L_r\varphi\}$ is not satisfiable, it follows that for each $i$, $\{\gamma_i\}\cup \{L_r\varphi\}$ is not satisfiable. So without loss of generality we may assume that $\gamma$ is of the form 
	$$ \gamma= \bigwedge_j p_{j} \wedge \bigwedge_{j'} \neg p_{j'} \wedge \bigwedge_l L_{r_{l}}\varphi_{l} \wedge \bigwedge_{l'} M_{s_{l'}}\varphi_{l'}.$$ 
	
 Associate  to	$L_{r_{l}}\varphi_{l}$  a linear inequality of the form $\sum_{i\in I_{\varphi_l}} x_i \geq r_l$. Likewise, for a formula $M_{r_{l'}}\varphi_{l'}$, consider the linear inequality of the form $\sum_{i\in I_{\varphi_{l'}}} x_i \leq s_{l'}$.
	
	Let $S_\Gamma$ be the set of all above inequalities together with 
	\begin{eqnarray*}
		x_i &\geq & 0 \ \ \ \  \ \ \ \ \ \ \ 1\leq i\leq n \\
		x_i &\leq & 1 \ \ \ \  \ \ \ \ \ \ \ 1\leq i\leq n  \\
		x_1+\dots +x_n &\geq & 1 \\
		x_1+\dots +x_n &\leq & 1.
	\end{eqnarray*}
	
	By Fact \ref{system}
	$S_\Gamma$ is solvable, since $\Gamma$ is satisfiable.
	Now consider the following optimization problem:
	\begin{eqnarray*}
		&\mbox{Maximize}& \ \ \   \sum_{i\in I_{\varphi}} x_{i}\\
		&\mbox{Subject to}& \ \ \  S_\Gamma.
	\end{eqnarray*}

	Since $S_\Gamma$ is solvable, it defines a non-empty closed and  bounded set of $\mathbb{R}^n$. Hence by the Fundamental theorem of linear programming (see Theorem 3.4 in \cite{van:lp01})  the above problem has a  solution.
	Let $M$ be the maximum value of $\sum_{i\in I_{\varphi}} x_{i}$. Since $\Gamma\cup\{L_r\varphi\}$ is not satisfiable, $S_{\Gamma}\cup   \{\sum_{i\in I_{\varphi}} x_{i}\geq r\}$ is not solvable. Hence it follows that $M<r$. So for every rational number $r'$ with $M<r'<r$,   one can see that  $\Gamma\cup\{L_{r'}\varphi\}$ is not satisfiable. 
	
	The other assertion can be shown similarly. 
\end{proof}

\begin{prop}\label{max}
Let $\Gamma$ be a finitely satisfiable positive theory. Then $\Gamma$ can be extended to a maximally finitely satisfiable positive theory. Furhermore, if $\Gamma$ is  maximally finitely satisfiable then  for each $PPL$-formula $\varphi$ and $r\in \mathbb{Q}\cap[0,1]$, $\Gamma$ contains at least one of the formulas $L_r\varphi$  and      $M_r\varphi$.

\end{prop}
\begin{proof}
Suppose that $\psi_1,\psi_2,\dots$ is an enumeration of positive formulas. Put $\Sigma_0=\Gamma$. For each $n\in \mathbb{N}$ define
\[  \Sigma_{n+1}= \left\{
\begin{array}{ll}
\Sigma_n\cup\{\psi_{n+1}\} & \text{if it is finitely satisfiable,}  \\
\Sigma_n &  \text{otherwise.} \  \end{array}
\right. \]
Let 	$\Sigma=\bigcup\Sigma_n$. Then it is easy to see that $\Sigma$ is maximally  finitely satisfiable positive theory. Now suppose that  $\Gamma$ is  maximally  finitely satisfiable and $L_r\varphi\not\in\Gamma$. So there is a finite subset $\Sigma'\subseteq \Gamma$ such that $\Sigma'\cup\{L_r\varphi\}$ is not satisfiable.
But this implies that $\Sigma'\models M_r\varphi$ and  then each finite subset $\Sigma''$ of $\Gamma$ including $\Sigma'$ has a model satisfying $M_r\varphi$. Therefore $\Gamma\cup \{M_r\varphi\}$ is finitely satisfiable and as $\Gamma$ is maximally finitely satisfiable, it follows that $M_r\varphi\in\Gamma$.
\end{proof}

Next proposition introduces examples of maximally finitely satisfiable $PPL$-theories and is needed for Theorem \ref{ppl compact}.

\begin{prop}\label{th(M)}
	Let  $(\mathfrak{M},w)$ be a finitely additive probability model. Then the positive theory of $(\mathfrak{M},w)$, i.e.
	$Th_+(\mathfrak{M},w)=\{\varphi\in PPL\ |\ \mathfrak{M},w\models\varphi \}$,
	is a maximally finitely satisfiable $PPL$-theory.
\end{prop}

\begin{proof}
	Suppose that $\Sigma$ is a finitely satisfiable $PPL$-theory containing $Th_+(\mathfrak{M},w)$.
	By induction on the complexity of formulas we can show that if $\phi\in\Sigma$ then $\phi\in Th_+(\mathfrak{M},w)$, for each $PPL$-formula $\phi$.
\end{proof}

\begin{thm}[$PPL$-Compactness]\label{ppl compact}
	Let $\Gamma$ be a finitely satisfiable  positive theory. Then $\Gamma$ has a finitely additive probability model.
\end{thm}

\begin{proof}	
  In the following we define the model $(\mathfrak{M},w_0)$ in a way that $\mathfrak{M},w_0\models\Sigma$.
	
	Let $\Omega_C$ be the set of all maximally finitely satisfiable sets of positive formulas. Put
	$$\Theta=\{[\varphi]\ |\; \varphi \text{ is a  positive formula} \}$$
	where
	$[\varphi]=\{w\in \Omega_C\ |\ \varphi\in w \}$.
	Note that the set $(\Theta,\cap,\cup, [\perp],[\top])$ forms a lattice and for every $\varphi,\psi\in PPL$, we have  the followings
	
	\begin{itemize}
		\item   $[\varphi]\cap [\psi]=[\varphi\wedge \psi]$.
		\item  $[\varphi]\cup [\psi]=[\varphi\vee \psi]$. 
	\end{itemize}
	
	Moreover, as $PPL$ is not closed under negation, $\Omega_C$ is not an algebra. Now we define the function 	
	$T': \Omega_C\times\Theta \rightarrow [0,1]$ as follows:
	$$T'(w)([\varphi]) = \sup\{r\in \mathbb{Q}\cap [0,1]\ |\ L_r\varphi\in w \}.$$
	
	\noindent
	\textbf{Claim.} For each $w\in \Omega_C$ the function $T'(w)$ is a valuation on the lattice $\Theta$.
	\begin{proof}[Proof of Claim.]
				Let $w\in \Omega_C$.
		\begin{itemize}
						
			\item $T'(w)(\emptyset)=0$, since  $L_r\perp$ is not satisfiable for any $r>0$. 
			\item Suppose that	$[\varphi]\subseteq[\psi]$. Then we have to prove that 	$T'(w)([\varphi])\leq T'(w)([\psi])$.
			We show that if $[\varphi]\subseteq[\psi]$, then we have $\varphi\models\psi$. Otherwise, there exists a model $\mathfrak{N},v \models \varphi$ and $\mathfrak{N},v \nmodels \psi$. So, $\varphi\in Th_+(\mathfrak{N},v)$ and $\psi\notin Th_+(\mathfrak{N},v)$, which is a  contradiction by Lemma \ref{th(M)}.

			\item We have to show that for all $[\varphi_1],[\varphi_2]\in \Theta$,
			$$T'(w)([\varphi_1]) + T'(w)([\varphi_2]) = T'(w)([\varphi_1]\cup[\varphi_2]) + T'(w)([\varphi_1]\cap[\varphi_2]).$$
		
			Suppose that 	$T'(w)([\varphi_i])=\alpha_i$, for $i=1,2$ and $T'(w)([\varphi_1\vee\varphi_2])=\alpha_\vee$
			and 		$T'(w)([\varphi_1\wedge\varphi_2])=\alpha_\wedge$.
			
			If 	$\alpha_1 +\alpha_2 <\alpha_\vee+\alpha_\wedge$,
			then find 		$\epsilon_1,\epsilon_2,\epsilon_\vee,\epsilon_\wedge>0$
			such that
			$\alpha_i'=(\alpha_i+\epsilon_i)\in\mathbb{Q}$ for $i\in\{1,2\}$,
			and
			$\alpha_i'=(\alpha_i-\epsilon_i)\in\mathbb{Q}$, for $i\in\{\wedge,\vee\}$,
			and 	
			$(\alpha'_1)+(\alpha'_2) < (\alpha'_\vee) +(\alpha'_\wedge)$.
			But in this case
			$\{M_{\alpha_1'}\varphi_1,\; M_{\alpha'_2}\varphi_2,\; L_{\alpha'_\vee}(\varphi_1\vee\varphi_2),\; L_{\alpha'_\wedge}(\varphi_1\wedge\varphi_2)\}$
			is a finite subset of $w$ and not satisfiable, a contradiction.
			
A similar argument shows that the inequality   $\alpha_1 +\alpha_2 >\alpha_\vee+\alpha_\wedge$ leads to a contradiction.
		\end{itemize}
	\end{proof}
	Now let $\mathcal{B}(\Theta)$ be the boolean algebra generated by $\Theta$. By Fact \ref{SHT thm} for every $w\in \Omega_C$, one can extend the valuation $T'(w)$ to a finitely additive measure $T''(w)$ on $\mathcal{B}(\Theta)$. Subsequently,  by Fact \ref{tarski fa} we can extend each $T''(w)$ to a finitely additive measure $T_C(w)$ on $\mathcal{P}(\Omega_C)$. 
Note that the measurability of $T_C:\Omega\times \mathcal{P}(\Omega_C)\rightarrow [0,1]$ comes for free.
Now to define the valuation function $v_C$, for each proposition $p$, put $v_C(p)=\{w\in \Omega_C\ |\ p\in w\}$.
	
Having defined functions $T_C$ and $v_C$, we assume the model $\mathfrak{M}_{C}=(\Omega_C,\mathcal{P}(\Omega_C),T_C,v_C)$.
The following claim characterizes the satisfaction relation of $\mathfrak{M}_{C}$.    

	\noindent
\textbf{Claim.} For every positive formula $\varphi$ and $w\in \Omega_C$, $$\mathfrak{M}_{C},w\models \varphi\;\;\;\  \text{if and only if}\;\;\;\;\ \varphi\in w.$$
 
The above claim states that inside $\mathfrak{M}_{C}$,  for each $\varphi$ we have $\llbracket\varphi\rrbracket_{\mathfrak{M}_{C}} = [\varphi]$.
	\begin{proof}[Proof of Claim.]
		The proof  proceeds by induction on the complexity of positive formulas. The induction base for  atomic formulas as well as the induction step for boolean operators $\wedge,\vee$ are clear.
		
		Now consider the case where $\varphi=L_r\psi$, knowing that
		$\llbracket\psi\rrbracket_{\mathfrak{M}_{C}}= [\psi]$. 
		Now suppose $w\in[L_r\psi]$. So, $L_r\psi\in w$ and $T_C(w)([\psi])=\sup\{\alpha\ |\ L_\alpha\psi\in w \} \geq r$.
		So by induction hypothesis $T_C(w)(\llbracket\psi\rrbracket_{\mathfrak{M}_{C}}) \geq r$. Conversely, suppose that
		$w\in \llbracket\varphi\rrbracket_{\mathfrak{M}_{C}}$. In this case by induction hypothesis we have $T_C(w)([\psi]) \geq r$.
		Therefore,  $\sup\{\alpha\ |\ L_\alpha\psi\in w \} \geq r$.
		Now if $L_r\psi\not\in w$, then, as $w$ is maximally finitely satisfiable, $w\cup \{L_r\psi\}$ is not finitely satisfiable. So there exists a finite subset $w'$ of $w$ such that $w'\cup \{L_r\varphi\}$ is not satisfiable. Thus, by  Lemma \ref{lemwp} there exists $r'<r$ such that
		$w'\cup \{L_{r'}\psi\}$
		is not finitely satisfiable and   $L_{r'}\psi\nin w$. But this contradicts with $\sup\{\alpha\ |\ L_\alpha\psi\in w \} \geq r$. Hence $\llbracket\varphi\rrbracket_{\mathfrak{M}_{C}}= [\varphi]$ and the induction is proved for $L_r\varphi$. Lemma \ref{lemwp} can be applied to show that the  induction hypothesis holds for $\varphi= M_s\psi$.
	\end{proof}
	
Now having proved the claim we can finish the proof by noticing that if $\Gamma$ is a finitely satisfiable theory, then by Lemma \ref{max} one can find a maximally finitely satisfiable $w$ which includes $\Gamma$. Hence we have that $  \mathfrak{M}_{C},w\models \Gamma$.
\end{proof}

\section{The L\"{o}wenheim-Skolem Number of Probability Logics}\label{LS sec}

In this section we study the L\"{o}wenheim-Skolem number of the class  of probability and  finitely additive probability models. The L\"{o}wenheim-Skolem number of a class of models $\mathcal{C}$  of a logic $\mathcal{L}$ is the least infinite cardinal $\kappa$ such that every satisfiable $\mathcal{L}$-theory  has a model of size at most $\kappa$. In this section we prove that this number is uncountable cardinal of at most $2^{\aleph_0}$ for the class of  probability models, while it is $\aleph_0$ for the class of finitely additive models.

\begin{thm}\label{LS PL}
Let $\lambda$ be the L\"{o}wenheim-Skolem number of probability models with respect to probability logic. Then $\aleph_0 <\lambda \leq 2^{\aleph_0}$.
\end{thm}
\begin{proof}
If theory $\Sigma$ is satisfiable then it is consistent. Hence by Theorem 3.2.13 of \cite{zhou:thesis07} there is a canonical model which models $\Sigma$. But the size of this model is $2^{\aleph_0}$. So, $\lambda \leq 2^{\aleph_0}$. Furthermore, the following example shows that there is a $BPL$-theory which does not have a countable model. Hence the proof is complete.
\end{proof}
\begin{exm}
		Let
		$$\Gamma = \{ L_{\frac{1}{2}}\neg(p_i \leftrightarrow p_j) \; |\; i<j, \; i,j\in\mathbb{N}  \}.$$
		
		$\Gamma$ is satisfiable, specially it has a model of size $2^{\aleph_0}$.
		To see this, define the model $\mathfrak{N}$ as follows. Let $\Omega_\mathfrak{N}=\{0,1\}^\mathbb{N}$ and
		$\mathcal{A}_\mathfrak{N}$ be a product $\sigma$-algebra, i.e. a $\sigma$-algebra generated by direct product $\prod_{i\in \mathbb{N}}A_i$ where except for a finite number of $A_i$s the rest of them are $\{0,1\}$.
		Suppose $\mu(0)=\mu(1)=\frac{1}{2}$ and
		$T((0_i)_{i\in\mathbb{N}})$ is a product measure of $\mu$. Moreover, for each proposition $p_j$ put
		$v(p_j)=\{(w_i)_{i\in\mathbb{N}}\, |\, w_j=1\}$. Therefore,
		\begin{align*}
		T((0_i)_{i\in\mathbb{N}})(\{(w_i)_{i\in\mathbb{N}}\; |\; \mathfrak{N},(w_i)_{i\in\mathbb{N}}\not\models p_l\leftrightarrow p_j\}) & =\\
		T((0_i)_{i\in\mathbb{N}})(\prod_{i\in \mathbb{N}} A_i) + T((0_i)_{i\in\mathbb{N}})(\prod_{i\in \mathbb{N}} B_i) & = \frac{1}{2^2} + \frac{1}{2^2} =\frac{1}{2}
		\end{align*}
		where $A_l=B_j=\{1\}$ and $A_j=B_l=\{0\}$ and $A_i=B_i=\{0,1\}$ otherwise. So $\mathfrak{N},(0_i)_{i\in\mathbb{N}}\models \Gamma$.	
		
		Now we show that there is no countable model for $\Gamma$. Suppose $\mathfrak{M},w\models\Gamma$ and
		$\Omega_\mathfrak{M}$ is a countable set. Specially we can assume that  $\mathcal{A}_\mathfrak{M}=\mathcal{P}(\Omega_\mathfrak{M})$. 
		Hence we have $\sum_{i=1}^\infty T(w)(\{w_i\})=1$. Therefore, there is $N\in\mathbb{N}$ such that
		$\sum_{i=1}^NT(w)(\{w_i\})>\frac{1}{2}$.
		On the other hand for each finite number of worlds, say $w_1,\dots,w_n$, there are $i,j\in\mathbb{N}$ such that 
		$\mathfrak{M}, w_k \models (p_i\leftrightarrow p_j)$, for $k=1,\dots,n$. Since
		$T(w)(\{w'\, |\, \mathfrak{M},w'\nmodels p_i\leftrightarrow p_j\})\geq \frac{1}{2}$, we should have
		$T(w)(\{w_1,\dots,w_N\})\leq T(w)(\{w'\, |\, \mathfrak{M},w'\models p_i\leftrightarrow p_j\})\leq  \frac{1}{2}$, a contradiction.
	\end{exm}

\begin{rem}
It is shown, for example in \cite{kozen2013strong} any satisfiable $PL$-theory has an analytic probability model. On the other hand, any analytic space is either countable or has size of continuum. Hence the L\"{o}wenheim-Skolem number of class of analytic probability models is $2^{\aleph_0}$.   	
\end{rem}

Now we turn to prove the L\"{o}wenheim-Skolem number of finitely additive models.

\begin{thm}\label{LS fad}
	The L\"{o}wenheim-Skolem number of the class of finitely additive probability models with respect to probability logic is $\aleph_0$. 
\end{thm}
\begin{proof}
	Let $\Gamma$ be a satisfiable theory and suppose  a finitely additive probability model $\mathfrak{M}=(M, \mathcal{B}, T:M\times \mathcal{B}\rightarrow [0,1],v)$ models $\Gamma$ at a point $w_0\in M$. We construct a countable finitely additive probability model $\mathfrak{M}'$ which models $\Gamma$.   
	To this end, define a countable set $\Omega\subseteq M$ which includes $w_0$ and has the property that for each probability formulas $\varphi$, if
	$\llbracket\varphi\rrbracket_\mathfrak{M}  \neq \emptyset$,
	then $\Omega\cap \llbracket\varphi\rrbracket_\mathfrak{M} \neq \emptyset$. For each formula $\varphi$ let $ \llbracket\varphi\rrbracket_{\Omega}=\Omega\cap \llbracket\varphi\rrbracket_\mathfrak{M}$. Put $\mathcal{B}_{\Omega}=\{\llbracket\varphi\rrbracket_{\Omega}\ |\ \varphi\in PL\}$. Note that, $\mathcal{B}_{\Omega}$ forms an algebra. Furthermore,  
	
	\noindent
	\textbf{Claim:} If $ \llbracket\varphi\rrbracket_{\Omega}= \llbracket\psi\rrbracket_\Omega$, then
	$\llbracket\varphi\rrbracket_\mathfrak{M} = \llbracket\psi\rrbracket_\mathfrak{M}$.
	
	\begin{proof}[Proof of Claim.]
		To see this, we may suppose that both sets $ \llbracket\varphi\rrbracket_{\Omega},$ $\llbracket\psi\rrbracket_\Omega$ are nonempty. Now if  	$\llbracket\varphi\rrbracket_\mathfrak{M}\neq  \llbracket\psi\rrbracket_\mathfrak{M}$, then either of the sets  $\llbracket\varphi\wedge \neg\psi\rrbracket_\mathfrak{M}$ and  $\llbracket\psi\wedge \neg\varphi\rrbracket_\mathfrak{M}$ are nonempty. Hence we have  $(\llbracket\varphi\wedge \neg \psi \rrbracket_{\Omega})\cup (\llbracket\psi\wedge \neg \varphi \rrbracket_{\Omega})\neq \emptyset $. But this implies that $ \llbracket\varphi\rrbracket_{\Omega}\neq \llbracket\psi\rrbracket_\Omega$.    
	\end{proof}

	Now define the function 	
	$T_\Omega: \Omega\times \mathcal{B}_{\Omega} \rightarrow [0,1]$ as follows:
	$$T_\Omega(v)( \llbracket\varphi\rrbracket_{\Omega}) = T_\mathfrak{M}(v)(\llbracket\varphi\rrbracket_\mathfrak{M}).$$
	By the above claim $T_{\Omega}$ is a well-defined function. It is not hard to see that for each $w\in\Omega$, $T_{\Omega}(w,-)$ defines a finitely additive probability measure.   Moreover, for each formula $\varphi$ and $r\in \mathbb{Q}\cap [0,1]$, $$\{w\in \Omega\ |\  T_{\Omega}(w, \llbracket\varphi\rrbracket_{\Omega})\geq r\}= \llbracket L_r\varphi\rrbracket_{\Omega}.$$ Hence $T$ is a measurable function.

	Finally, for each proposition $p$, put 	
	$v_\Omega(p)= v_\mathfrak{M}(p) \cap \Omega_\mathfrak{N}$ and set $\mathfrak{N}=(\Omega,\mathcal{B}_{\Omega},T_{\Omega},v_{\Omega})$.
	
	By induction on the complexity of formulas one can prove that
	$\llbracket\varphi\rrbracket_\mathfrak{N}=  \llbracket\varphi\rrbracket_\Omega$.
	Therefore,
	$\mathfrak{N},w\models\Gamma$.
\end{proof}

\section{Conclusion}

In this paper we investigated probability logic from model theoretic point of view. 
Specifically we study the compactness property and the L\"owenheim-Skolem number of probability logic with respect to both of the class of probability models and finitely additive models.
We showed that, although probability logic does  not have the compactness property the basic and positive fragments of that are compact respectively to the class of probability models and finitely additive probability models. 
Furthermore, we proved that the L\"owenheim-Skolem number of probability logic is $2^{\aleph_0}$ while it is $\aleph_0$ when we consider finitely additive models.

One of the other interesting issue in model theory which is worthwhile to study for probability logic is the Lindst\"om type theorem.
The Lindstr\"om type theorems characterize logics in terms of model theoretic concepts.
In 1969 Lindstr\"{o}m proved that first-order logic has the maximal expressive power among the abstract logics containing it with the compactness and the L\"{o}wenheim-Skolem properties.
This kind of characterization is widely studied for other logics specially for modal logics, for example \cite{rijke:lind95,  benthem:lind09, otto:lind08, enq:general13,zoghifard2018first}.  In addition to the compactness, bisimulation invariance property used to prove a characterization theorem for modal logic. 
Bisimulation of Markov processes is widely studied in many literature and some kind of definitions are given for them, \cite{deshar:bisim02,Deshar:approx03,danos2006bisimulation}.
In all versions of Lindstr\"om's style theorems the compactness property plays an essential rule. Kurz and Venema in \cite{venema:coalg10} asked whether one can give a version of Lindstr\"om theorem for non-compact logic such as probability logic.

Since studying probability logic from coalgebraic perspective is significant in computer science, investigating the problems of this paper and finding an appropriate version of Lindstr\"om's theorem for this logic can be a good guide for giving a general version of Lindstr\"om's theorem for non-compact logic.

One of the other valuable issues is to study first-order probability modal logic. There are a few literature considering some versions of first-order probability logic, see \cite{halpern:analysis90,savic2017first}.

\noindent
\textbf{Acknowledgments.} 
Part of results of this paper is presented in a short talk given in Advances in Modal Logic 2018 \cite{pourmahdian2018compactness}. The authors would like to thank to anonymous referees for giving some instructive comments which have been useful in proving our results.

\appendix
\section{Appendix}

In this part  some basic notions and results from measure theory, used in this paper, are reviewed. For further reading on measure theory see \cite{folland:real13}.

 Recall that a family $\mathcal{A}$ of subsets of a non-empty set $\Omega$ is called a \textit{lattice} if $\emptyset,\Omega\in \mathcal{A}$ and it is closed under finite unions and intersections. If, furthermore,  $\mathcal{A}$  is closed under complements then it is a \textit{boolean algebra} (or simply an algebra). Call an algebra $\mathcal{A}$ a $\sigma$-\textit{algebra} provided that if it is closed under countable unions. For a collection $\mathcal{A}$ of subsets of $\mathcal{P}(\Omega)$,
there exists a $\sigma$-algebra $\sigma(\mathcal{A})$ generated by $\mathcal{A}$ which is the intersection of all $\sigma$-algebras containing $\mathcal{A}$. While $\mathcal{P}(\Omega)$ is an obvious example of a $\sigma$-algebra over set $\Omega$, for a topological space $(\Omega,\tau)$ the family $\mathcal{C}$ of closed subsets of $\Omega$ forms a lattice. Furthermore members of $\sigma(\mathcal{C})$ are called Borel subsets of $\Omega$.

Moreover, $\mathcal{A}\subseteq \mathcal{P}(\Omega)$ is called a
\textit{monotone class} if it is closed under  unions of countable increasing sequences and also intersections of countable  decreasing sequences.
\begin{fact}[Monotone class]\label{monotone class}(Lemma 2.35 in \cite{folland:real13})
	The monotone class generated by an algebra $\mathcal{A}$ is equal to $\sigma(\mathcal{A})$. 
\end{fact}

A \textit{measurable space} is a pair $(\Omega,\mathcal{A})$ where $\mathcal{A}$ is a $\sigma$-algebra on the non-empty set $\Omega$. Each $A\in \mathcal{A}$ is named a measurable set.
For two measurable spaces $(X,\mathcal{A})$ and $(Y,\mathcal{B})$ the function
$f:X\rightarrow Y$ is a \textit{measurable function} if $f^{-1}(B)\in\mathcal{A}$ for each $B\in\mathcal{B}$.

Let $\mathcal{A}$ be a lattice over $\Omega$. Then a non-negative real-valued set function
$\mu:\mathcal{A}\rightarrow \mathbb{R}$ is a \textit{valuation} if it satisfies the following conditions:
\begin{itemize}
	\item (Strictness) $\mu(\emptyset)=0 $,
	\item (Monotonicity) if $A\subseteq B$ is in $\mathcal{A}$, then $\mu(A)\leq \mu(B)$,
	\item  (Modularity) $\mu(A)+\mu(B)=\mu(A\cup B)+\mu(A\cap B)$, for all $A, B\in\mathcal{A}$.
\end{itemize}
	
	In case $\mathcal{A}$ is an algebra then  the function
	$\mu:\mathcal{A}\rightarrow \mathbb{R}$ is called a \textit{finitely additive measure}. In this situation modularity implies monotonicity.   Furthermore, $\mu$ is \textit{premeasure} whenever for any $\{A_i\}_{i\in\mathbb{N}}$  of pairwise disjoint members of $\mathcal{A}$ if $\bigcup_{i\in\mathbb{N}}A_i\in \mathcal{A}$, then $\mu(\bigcup_{i\in\mathbb{N}} A_i)=\sum_{i\in\mathbb{N}}\mu(A_i)$. 

	Finally  for a $\sigma$-algebra $\mathcal{A}$ a premeasure 
	$\mu: \mathcal{A}\rightarrow \mathbb{R}$ is called a  \textit{$\sigma$-additive measure}.

A finitely or a $\sigma$-additive measure $\mu$ is a probability measure when $\mu(\Omega)=1$. For brevity,  a $\sigma$-additive measure is simply called a measure.

\begin{fact}\label{premeas equiv}
	Let $\mathcal{A}$ be an algebra over $\Omega$. 	A function 
	$\mu:\mathcal{A}\rightarrow \mathbb{R}$ is a premeasure if for each decreasing sequence
	$A_0\supseteq A_1\supseteq \dots$
	of elements of $\mathcal{A}$, if
	$\bigcap_{i} A_i=\emptyset$
	then 
	$\lim_{i\rightarrow\infty} \mu(A_i)=0$.
\end{fact}

A \textit{measure space} is a triple $(X,\mathcal{A},\mu)$ where $\mu$ is a measure on the $\sigma$-algebra $\mathcal{A}$.

The following standard fact states how to extend a valuation over a lattice $\mathcal{L}$ to a finitely additive measure over algebra $\mathcal{B}(\mathcal{L})$ generated by $\mathcal{L}$.    

\begin{fact}[Smiley--Horn--Tarski Theorem in \cite{ghk:conla03}]\label{SHT thm}
	Let $\mu$ be a valuation defined on a lattice  $\mathcal{L}$ of subsets of $X$. Then $\mu$ can be uniquely  extended to a finitely additive measure $\mu^*$ on the algebra $\mathcal{B}(\mathcal{L})$ generated by $\mathcal{L}$.
	
\end{fact}

The following facts can also be shown using Carath\'{e}odory's extension theorem.  

\begin{fact}(Theorem 1.22 in \cite{horntarski:measure48})\label{tarski fa}
	Let $\mu$ be a finitely additive measure on a boolean algebra $\mathcal{A}$ of $X$.  Then $\mu$ could be extended to a finitely additive measure on any boolean algebra $\mathcal{A}'$ containing $\mathcal{A}$.
\end{fact}

\begin{fact}\label{caratheo}(Theorem 1.14 in \cite{folland:real13})
	Let $\mu$ be a finite premeasure on boolean algebra $\mathcal{A}$. Then $\mu$ has a unique extension to $\mu^*$ on $\sigma(A)$.
\end{fact}


\end{document}